\documentclass[final,leqno,onefignum,onetabnum]{siamltex}

\usepackage{subfigure,upgreek,dsfont,braket,amssymb,amsmath}
\usepackage[small,bf]{caption}
\usepackage{booktabs,dsfont}
\usepackage{algorithm}
\usepackage{float}
\usepackage{graphicx}
\usepackage{algorithmicx}
\usepackage{color,float}

\usepackage{amssymb, amsmath}

\newcommand{\al}{\alpha}
\newcommand{\be}{\beta}
\newcommand{\ga}{\gamma}
\newcommand{\de}{\delta}
\newcommand{\la}{\lambda}

\newcommand{\auskommentieren}[1]{}

\newcommand{\beq}{\begin{equation}}
\newcommand{\eeq}{\end{equation}}

\DeclareMathOperator{\dive}{div}
\DeclareMathOperator{\curv}{curv}
\DeclareMathOperator{\dist}{dist}

\DeclareMathOperator{\scale}{scale}

\title{Numerical approximation of positive power curvature flow 
via deterministic games.}

\author{Heiko Kr\"oner\thanks{Fachbereich Mathematik, Universit\"at Hamburg, Bundesstra\ss e 55, 20146 Hamburg, Germany.
{\tt  heiko.kroener@uni-hamburg.de}}}

\begin{document}
\maketitle
\slugger{mms}{xxxx}{xx}{x}{x--x}

\begin{abstract}
We approximate the level set solution for the motion of an embedded closed curve in the plane with normal speed $\max(0, \kappa)^{\ga}$ where $\kappa$ is the curvature of the curve  and $\frac{1}{3}<\ga<1$ by the value functions of a family of deterministic two person games. We show convergence of the value functions to the viscosity solution of the level set equation and propose a numerical scheme for the calculation of the value function. We illustrate the convergence properties of the scheme for different parameter values in some example cases.
\end{abstract}


\section{Introduction}
\label{intro}
In \cite{KS} R. Kohn and S. Serfaty present a game theoretic approach to the curve-shortening flow, i.e. the mean curvature flow  of  curves in the plane.
The (time-dependent) level set formulation of this flow is a degenerate 
parabolic, possibly singular equation and is interpreted as the limit of the value functions of a family of discrete-time, 
two-person games. It is proved that these value functions converge to the viscosity solution of the level set equation,
cf. \cite[Theorem 2]{KS}. Furthermore, using a
'minimal exit time game' a rate for the convergence of the value functions
of this game to the time-independent level set formulation of the (mean) curvature flow 
is proved in \cite{KS}. 
We note, that the corresponding convergence result (without rates) for the positive curvature 
flow (here the normal speed is the curvature if it is nonnegative and zero otherwise) is given in \cite[Theorem~5]{KS}. In contrast to the curve shortening flow not much is known about the behavior of a non convex initial curve in the plane moving by positive curvature flow. But the latter has been used for image processing, see the remarks in \cite[page 16]{KS} and \cite{MalladiSethian1996}.
In form of examples modified flows are presented in \cite[Section 1.7]{KS} which have a game theoretic interpretation. But proofs of
convergence for the value functions of the corresponding games for these modified examples are not given therein.
Among these is a game, cf. \cite[Example 3, Section 1.7]{KS}, which can handle the case 
of the curvature flow of a convex curve in the plane at which the normal speed is given -- instead of the curvature -- 
by a power $\ga>0$ of the (mean) curvature. We call this power curvature flow (PCF) (and the analogous flow 
for hypersurfaces power mean curvature flow (PMCF)). PCF is only a well-defined flow when the initial curve is convex. 
In \cite{KS} five open questions are formulated, cf. \cite[Section 1.8]{KS}; among these is the
question of numerical advantages or disadvantages of this approach.

In this paper we consider for $\frac{1}{3}<\ga<1$ the more general, so-called positive power curvature flow PPCF for which the normal speed is given by $\kappa_+^{\ga}$ (where $\kappa_+=\max(0, \kappa)$) and which coincides with PCF if the initial curve is convex.  Our aim is to approximate this flow by a family of value functions of suitable games. We show convergence of the value functions to the unique viscosity solution of the time-dependent level set formulation of  PPCF 
and present a numerical algorithm for this approach.
We demonstrate the convergence of the algorithm for some example cases and analyze the effect of different
choices for the numerical parameters.

Let us relate the content of our paper to \cite{KS}. There is no game in \cite{KS} for PPCF. The game in \cite[Example 3, Section 1.7]{KS} is suggested therein to approximate PCF in case $0<\ga<1$. The corresponding limit equation \cite[(1.23)]{KS} has the (formal) disadvantage  that its left-hand side becomes $-\infty$ if there are non convex zero-level sets (in case they are curves) of the continuum solution (e.g. at the initial time outside of the initial curve).  By considering PPCF we avoid this convexity issue, especially we can use non convex initial curves. The game we use for the approximation of PPCF is essentially the one in  \cite[Example 3, Section 1.7]{KS} but  it turns out that it is necessary (at least for our proof) to use a slightly different definition of the value function. The difference is that we
take the minimum in the definition of the value function over a ($\epsilon$-dependent) compact subinterval
of $(0, \infty)$ instead over $(0, \infty)$, cf. (\ref{40}). This difference is not due to the fact that we consider PPCF instead of PCF but is of general nature. The initial and endpoints of the subintervals $I_{\epsilon}$ have to satisfy some technical relations, cf. (\ref{86})-(\ref{401}), which imply that we need the restriction $\frac{1}{3}<\ga<1$. The interesting point is that $\frac{1}{3}$ is exactly the critical value at which the convergence behavior of  PCF changes, see Theorem \ref{c1} and the following remarks. 

Concerning a numerical application of the Kohn-Serfaty approach to 
the curve shortening flow (i.e. $\gamma=1$)
we refer to \cite{CarliniFalconeFerretti2010}.
In addition therein a semi-Lagrangian approximation of the curve
shortening flow is considered.
A semi-Lagrangian approach for the affine curvature flow case which corresponds 
to $\ga=\frac{1}{3}$ is content of 
\cite{CarliniFerretti2013}. 

We give some references for the numerical
approximation of flows by powers of the curvature (which do not 
use the Kohn-Serfaty approach).
A numerical approximation of the stationary level set solution of 
PCF in case $\ga \ge 1$  
is content of our two previous
papers \cite{K} and \cite{KKK}. There we approximate the regularized stationary level set equation 
(used in \cite{S} to prove existence of the stationary level set PMCF) 
by finite elements and prove rates for the 
approximation errors. 
The stationary level set equation of mean curvature flow
is also content of \cite{Carlini_Ferretti_2008}.

There are only few numerical papers which deal with PCF for $\gamma\neq 1$.

In \cite{BGN} and \cite{BGN2} stability bounds for a finite element approximation
of PCF for curves is derived, see also \cite{BGN3}. In \cite{MSe} the 
parametric formulation of the evolution
of plane curves driven by a nonlinear function of curvature and anisotropy
is considered.
Concerning PCF from the perspective of image processing we refer to 
\cite{AGLM} \cite{AM}, \cite{ST}, \cite{MS} and \cite{MalladiSethian1996}.

For estimates of the regularization error for the instationary level set 
formulation of mean curvature flow we refer to \cite{M} and \cite{D}.
For more general references to the numerical treatment of motion by 
curvature we refer to 
 \cite{M},  \cite{DDE}, \cite{D}, \cite{NV}, \cite{DD}
 , \cite{Deckelnick_Dziuk_2001}, \cite{OS} and \cite{Walkington} 
 where the instationary level set formulation is considered.

Our paper is organized as follows. In the remaining part of the present section we formulate
some known results about the behavior of curves evolving under PCF. In Section \ref{sec:1} we define
the value functions $u^{\epsilon}$. In Section \ref{sec:2} we show convergence of the value functions to the 
solution of the continuous problem. In Section \ref{sec:3} 
we present the algorithm in order to calculate $u^{\epsilon}$. Its application in some example cases as well as the effect of the choice
of different numerical parameters
is content of Section \ref{sec:4}.

As for the positive curvature flow little is known about the behavior of PPCF for not convex initial curves.
We state some well-known facts about the behavior of curves evolving under PCF in case $0<\ga \le 1$, see 
\cite{A03} and \cite{Andrews}, and remark that in these references also the case $\ga >1$ is treated.
\begin{theorem} \label{82}
 Let $\ga=1$ and $\Gamma_0$ a smooth, embedded closed curve given by an embedding
 $x_0 : S^1\rightarrow \mathbb{R}^2$. Then there exists a unique smooth solution $x: S^1 \times [0, T) \rightarrow
 \mathbb{R}^2$ of
 \beq \label{83}
 \frac{d}{dt}x = - \kappa^{\ga} \nu \quad (= - \kappa \nu)
 \eeq
 where $\nu$ denotes the outer unit normal, with initial data $x_0$. $\Gamma_t=x(S^1, t)$ converges
 to a point $p \in \mathbb{R}^2$ as $t \rightarrow T$. The rescaled curves $(\Gamma_t-p)/\sqrt{2(T-t)}$
 converge to the unit circle about the origin as $t \rightarrow T$.
\end{theorem}

The assumption of smoothness of the initial curve in Theorem \ref{82} can be relaxed to allow $\Gamma_0$ to be the 
boundary of a bounded open convex region and the curves $\Gamma_t$ are converging to $\Gamma_0$ in Hausdorff 
distance as $t \rightarrow 0$, see \cite{A1996}.

If $\ga = \frac{1}{3}$ and $\Gamma_0$ a smooth, convex, embedded closed curve then Theorem \ref{82} still
holds with a different rescaling, namely $(\Gamma_t-p)/(4(T-t)/3)^{\frac{3}{4}}$ converges to an ellipse of enclosed area
$\pi$ centered at the origin. The regularity assumption  on the initial curve can be weakened to
allow boundaries of open bounded convex regions.
For $\ga \in [\frac{1}{3}, 1]$ we get the following result.
\begin{theorem} \label{c1}
If $\ga \in [\frac{1}{3}, 1]$ and $\Gamma_0$ the boundary curve of an open bounded convex set in 
$\mathbb{R}^2$
then (\ref{83}) has a smooth and strictly convex solution for $0<t<T$ which converges to a point 
$p \in \mathbb{R}^2$ and the rescaled curves 
$(\Gamma_t-p)/(T-t)^{\frac{1}{1+\ga}}$ converge smoothly to a limit curve $\Gamma_{\ga}$ which
satisfies $\kappa^{\ga} = \la \left<x, \nu\right>$ for some $\la>0$.
\end{theorem}

Curves $\Gamma_{\la}$ are called homothetic solutions and in case $\ga>\frac{1}{3}$ all homothetic solutions
are circles. For $\ga=\frac{1}{3}$ the only homothetic solutions are ellipses. For $0<\ga<\frac{1}{3}$ a classification
of homothetic solutions is given in \cite[Theorem 1.5]{A03}.

There are results for non-convex initial curves in \cite{AST} and \cite[Chapter 8]{CZ2000}. 

For $0<\ga<\frac{1}{3}$ the isoperimetrical ratio becomes unbounded
near the final time for generic symmetric initial data, cf. the remarks in \cite[Section 1]{A03}.

There are results on PMCF in case of $n$-dimensional hypersurfaces in Euclidean space, 
$n \ge 2$, see for example \cite{S2} and \cite{SS} for the parametric formulation and \cite{S} for the time-independent 
level set formulation in case $\ga \ge 1$ and \cite{M} (and references therein) for the time-dependent level
set formulation of PCF in case $0<\ga \le 1$. 

Concerning a time-dependent level set equation for PCF we only know the equation
\begin{equation} \label{81}
\begin{aligned} 
u_t =& |Du| \left\{\dive  \left(\frac{Du}{|Du|}\right)\right\}^{\ga} \\
=& \left\{\sum_{i,j}\left(\de_{ij}-\frac{D_iuD_ju}{|Du|^2}\right)D_iD_j u\right\}^{\ga}|Du|^{1-\ga}
\end{aligned}
\end{equation}
in case $\ga=\frac{1}{2n-1}$, $n \in \mathbb{N}$, which is a little bit different from PCF since here negative curvatures are allowed.  
For PPCF the time-dependent level set equation is given by
\begin{equation} \label{81_}
\begin{aligned} 
u_t =& |Du| \left\{\dive  \left(\frac{Du}{|Du|}\right)\right\}_+^{\ga} \\
=& \left\{\sum_{i,j}\left(\de_{ij}-\frac{D_iuD_ju}{|Du|^2}\right)D_iD_j u\right\}_+^{\ga}|Du|^{1-\ga}
\end{aligned}
\end{equation}
for general $0<\ga\le 1$. We consider these equations in  $\mathbb{R}^2\times (0, \infty)$
with initial condition $u(\cdot, 0)=u_0$. We assume that $u_0 \in C^{0,1}(\mathbb{R}^2)$, $|u_0|\le 1$ and 
$u_0(x)=1$ for $x \in \mathbb{R}^2\backslash B_R(0)$ where $R>0$ is a sufficiently large radius.

Equations (\ref{81}), (\ref{81_}) are fundamental equations in image processing, 
especially in the case $\ga=\frac{1}{3}$ where these equations are the so-called 
affine curvature equations, cf. \cite{AGLM} and \cite {Giga}. 
Equations (\ref{81}), (\ref{81_}) are included by the general 
geometric equations which are dealt in \cite{CGG1}, 
\cite{GG1992} and  \cite{Giga} and therefore for each equation 
exists a unique viscosity solution.
We mention \cite{M} where convergence with rates of a regularization of (\ref{81}) is proved. 

\section{Game theoretic interpretation and definition of $u^{\epsilon}$}
\label{sec:1}
Let $\Omega$ be a smooth, bounded domain in the plane and $T>0$.
Let $u_0$ be a Lipschitz continuous function in the plane with $u_0<0$ 
in $\Omega$, $u_0\ge 0$ in $\mathbb{R}^2\backslash \bar \Omega$ and $|u_0|\le 1$. 
We assume $\frac{1}{3}<\ga<1$ and consider the following continuum equation
\beq \label{3}
\begin{cases}
u_t - |\nabla u|\varphi(\curv(u)) &=0 \\ 
u(\cdot, T) &= u_0
\end{cases}
\eeq
where 
\beq \label{13}
\curv(u) = - \dive(\nabla u/|\nabla u|)
\eeq
 is the curvature of the level set and
 \begin{equation} \label{12}
 \begin{aligned}
  \varphi(\kappa) = \begin{cases}
                     -|\kappa|^{\ga}, \quad &\kappa \le 0 \\
                     0, \quad &\kappa>0.
                     \end{cases} 
 \end{aligned}
\end{equation}
We note, that in this sense
the curvature of the unit circle is $-1$ because in $\{x\in \mathbb{R}^{n+1}: |x|-1=0\}$ we have $\curv(|x|)=-\dive \left(\frac{x}{|x|}\right)=
 -n<0$ in view of (\ref{13}). Furthermore, (\ref{3}) is not PPCF but
 after transforming the time variable $t \mapsto T-t$ we get the correct equation, i.e. PPCF (\ref{81_}). 
 As already explained above our equation (\ref{3}) differs from the continuum equation in \cite[Section 1.7, Example 3]{KS} in the fact that the latter requires $u_t=\infty$ if $\curv(u)>0$.

The goal is to approximate (\ref{3})
by the value functions of certain  deterministic, time-discrete two person games.  
For the convergence proof it will be crucial that
for $\kappa\le 0$ we can write
\beq \label{84}
\varphi(\kappa) = \sup_{s > 0}\left(\frac{1}{2}\kappa s^2-f(s)\right)
\eeq
where
\beq
f(s) = c_{\ga}s^{\frac{2\ga}{\ga-1}}, \quad
c_{\ga} = (1-\ga) (2 \ga)^{\frac{\ga}{1-\ga}}.
\eeq
The value functions $u^{\epsilon}$, $\epsilon >0$, are defined for $x \in \Omega$ and $0\le t \le T$ by
 \beq \label{40}
 u^{\epsilon}(x,t) = \min_{\|v\|=1,s \in I_{\epsilon}} \max_{b, \beta\in\{-1,1\}}
 u^{\epsilon}(x+b \epsilon s v + \beta \epsilon^2 f(s)v^{\perp}, t + \epsilon ^2)
 \eeq
 where 
 \beq \label{86}
 I_{\epsilon} = [\epsilon^{\al_1}, \epsilon^{-\al_2}]
 \eeq 
 with 
 \beq \label{601}
 \frac{1-\ga}{2\ga}<\al_1< \min \left(1, \frac{1-\ga}{\ga}\right) 
 \eeq
and
 \beq \label{401}
 0<\al_2 < \min\left( \al_1\frac{2\ga}{1-\ga}-1, \frac{1}{3}\right). 
 \eeq
 The difference from our definition to the definition in \cite[Example 3, Section 1.7]{KS} is that we minimize over $s \in I_{\epsilon}$ instead of $s>0$.
The first inequality in (\ref{601}) implies that the RHS of (\ref{401}) is positive so that 
there exists a corresponding $\al_2$. The $\al_1<1$ part of the second inequality in (\ref{601}) is needed for Case (1) in part (ii) 
of the proof of Lemma \ref{d1}. The other part of the second inequality in (\ref{601}) implies 
that the step size is at most a fixed positive power of $\epsilon$. The first part of the second 
inequality in (\ref{401}) is needed for part (i) of the proof of Lemma \ref{d1} while the second 
part is used for Case (2ba) of part (ii) of the proof of Lemma \ref{d1}. Inequality  (\ref{601}) 
implies $\ga >\frac{1}{3}$ which is the reason why we have only a convergence result for this case.

We can calculate $u^{\epsilon}$ numerically by setting $u^{\epsilon}(\cdot, T)=u_0$ and 'playing the game backwards in time', i.e. we calculate $u^{\epsilon}(\cdot, T-\epsilon^2)$, etc.. 

\section{Convergence of the value functions}
\label{sec:2}
In this section we show that the value functions converge to a solution of (\ref{3}).
For completeness we give the following definition.

\begin{definition}
 (i) A lower semicontinuous function $u$ is a viscosity supersolution of (\ref{3}) if whenever
 $\phi(x,t)$ is smooth and $u - \phi$ has a local minimum at $(x_0, t_0)$ we have
 \beq
 \phi_t -|\nabla \phi|\varphi(\curv(\phi)) \le 0
 \eeq
 at $(x_0, t_0)$ if $\nabla \phi(x_0, t_0) \neq 0$, and 
 \begin{equation}
 \begin{aligned}
 \phi_t \le 0 
  \end{aligned}
 \end{equation}
 at $(x_0, t_0)$ if $\nabla\phi(x_0, t_0)=0$.
 
 (ii) An upper semicontinuous function $u$ is a viscosity subsolution of (\ref{3}) if whenever
 $\phi(x,t)$ is smooth and $u - \phi$ has a local maximum at $(x_0, t_0)$ we have
 \beq
 \phi_t -|\nabla \phi|\varphi(\curv(\phi)) \ge 0
 \eeq
 at $(x_0, t_0)$ if $\nabla \phi(x_0, t_0) \neq	0$, and 
 \begin{equation}
 \begin{aligned}
 \phi_t \ge 0 
  \end{aligned}
 \end{equation}
 at $(x_0, t_0)$ if $\nabla\phi(x_0, t_0)=0$.

\end{definition}
We define 
\begin{equation}
 \begin{aligned}
  \bar u(X) =& \limsup_{Y\rightarrow X, \epsilon \rightarrow 0}u^{\epsilon}(Y) \\
  \underline u(X) =& \liminf_{Y \rightarrow X, \epsilon\rightarrow 0}u^{\epsilon}(Y).
 \end{aligned}
\end{equation}
$\bar u$ is upper semicontinuous, $\underline u$ is lower semicontinuous and 
$\underline u \le \bar u$
We will show that $\bar u$ is a viscosity subsolution of (\ref{3}) and that $\underline u$ is a viscosity 
supersolution of (\ref{3}). Then from a comparison principle we get that $\bar u=\underline u$ is a solution and that $u^{\epsilon}\rightarrow u$. The proof strategy starts with an argumentation adapted from 
\cite[Section 4.3]{KS} and then we distinguish several cases to handle our flow. 
\begin{lemma}
 $\bar u$ is a viscosity subsolution of (\ref{3}).
\end{lemma}
\begin{proof}
 We argue by contradiction. If not, then there is a smooth $\phi$ such that $(x_0, t_0)$ is
 a local maximum of $\bar u - \phi$ with
 \beq \label{20}
 \phi_t -|\nabla \phi|\varphi(\curv(\phi)) \le \theta_0 < 0 \text{ at }(x_0, t_0)
 \eeq
 if $\nabla \phi(x_0, t_0)\neq 0$ or we have $\nabla \phi(x_0, t_0)=0$ and 
 \beq \label{87}
 \phi_t \le \theta_0 <0 
 \eeq
 in $(x_0, t_0)$.
Adding to $\phi$ a nonnegative function whose derivatives at $(x_0, t_0)$ are all zero up
to second order, we can assume that $(x_0, t_0)$ is a strict local maximum of $\bar u-\phi$
in a $\de$-neighborhood of $(x_0, t_0)$. Let $(x^0_{\epsilon}, t^0_{\epsilon})\rightarrow 
(x_0, t_0)$ such that $u^{\epsilon}(x^0_{\epsilon}, t^0_{\epsilon})\rightarrow \bar u(x_0, t_0)$.

(i) Let us consider the case $\nabla \phi(x_0, t_0)\neq 0$.
We may assume that $\nabla \phi \neq 0$ in the above $\de$-neighborhood. We construct the following 
sequence
\begin{equation} \label{25}
\begin{aligned}
X^{\epsilon}_0 =& (x^0_{\epsilon}, t^0_{\epsilon}) \\
X^{\epsilon}_1 =& \left(x^0_{\epsilon}+b_0 \epsilon s_0 \frac{\nabla^{\perp} \phi(X^{\epsilon}_0)}{
|\nabla\phi(X^{\epsilon}_0)|}+\beta_0 \epsilon^2f(s_0)\frac{
\nabla\phi(X^{\epsilon}_0)}{|\nabla\phi(X^{\epsilon}_0)|}, t^0_{\epsilon}+{\epsilon}^2\right) \\
X^{\epsilon}_{k+1} =& X^{\epsilon}_k+\left(b_{k} \epsilon s_{k} \frac{\nabla^{\perp} \phi(X^{\epsilon}_k)}{
|\nabla\phi(X^{\epsilon}_k)|}+\beta_{k} \epsilon^2f(s_{k})\frac{
\nabla\phi(X^{\epsilon}_k)}{|\nabla\phi(X^{\epsilon}_k)|}, {\epsilon}^2\right)
\end{aligned}
\end{equation}
where the $b_k$ and $\beta_k$ are maximizing 
$u^{\epsilon}(X^{\epsilon}_{k+1})$
for given $s_k$ 
and $s_k\in I_{\epsilon}$  is chosen as will become clear in the following.
Then $u^{\epsilon}(X^{\epsilon}_k)\le 
u^{\epsilon}(X^{\epsilon}_{k+1})$ in view of (\ref{40}) and hence
\beq\label{21}
u^{\epsilon}(X^{\epsilon}_0) \le u^{\epsilon}(X^{\epsilon}_k).
\eeq
We let $X(s)$ be the continuous path that affinely interpolates between these points i.e. $X(t)
= X^{\epsilon}_k + (\frac{t-k\epsilon^2-t^0_{\epsilon}}{\epsilon^2})(X^{\epsilon}_{k+1}-X^{\epsilon}_k)$ for $t^0_{\epsilon}+k \epsilon^2
\le t \le t^0_{\epsilon}+(k+1)\epsilon^2$, and write $(x(t),t)=X(t)$.
We have
\begin{equation}
\begin{aligned}
\frac{\partial}{\partial t}\phi(X(t)) =& \partial_t\phi(x(t), t) + \nabla \phi(x(t), t) \cdot d_tx(t) \\
\frac{\partial}{\partial t}\nabla \phi(x(t), t) =& \nabla \phi_t(x(t),t) + D^2\phi(x(t),t)d_tx(t)
\end{aligned}
\end{equation}
and since 
\beq \label{110}
d_tx(t) = \frac{b_ks_k}{\epsilon}\frac{\nabla^{\perp}\phi(X^{\epsilon}_k)}{|\nabla\phi(X^{\epsilon}_k)|}
+\beta_k f(s_k) \frac{\nabla \phi(X^{\epsilon}_k)}{|\nabla \phi(X^{\epsilon}_k)|}
\eeq
is constant we get
\beq
\frac{\partial^2}{\partial t^2}\phi(X(t)) = \partial_t^2\phi(X(t)) + 2 \nabla \phi_t(X(t)) \cdot d_tx(t)
+\left<D^2\phi(X(t))d_tx(t), d_tx(t)\right>.
\eeq
Taylor expansion of $t \mapsto \phi(X(t))$ at $t^0_{\epsilon}+k\epsilon^2$ gives
\begin{equation} \label{260}
\begin{aligned}
\phi(X^{\epsilon}_{k+1}) -& \phi(X^{\epsilon}_k)  \\
=& \epsilon^2\frac{d}{dt}\phi(X(t^0_{\epsilon}+k\epsilon^2)) + \frac{\epsilon^4}{2}
\frac{d^2}{dt^2}\phi(X(t^0_{\epsilon}+k\epsilon^2)) \\
=& \epsilon^2\left(\partial_t\phi(X^{\epsilon}_k)+ \nabla\phi(X^{\epsilon}_k) d_tx(t) \right)\\
&+\frac{\epsilon^4}{2}\left(\partial_t^2\phi(X^{\epsilon}_k) + 2 \nabla \phi_t(X^{\epsilon}_k)d_tx(t)+
\left<D^2\phi(X^{\epsilon}_k)d_tx(t), d_tx(t)\right>\right)\\
&+o(\epsilon^2) \\
=& \epsilon^2\left(\partial_t\phi(X^{\epsilon}_k)+ |\nabla\phi(X^{\epsilon}_k)|\beta_k f(s_k) \right) \\
&+ \frac{\epsilon^2}{2}b_k^2s_k^2\left<D^2\phi(X^{\epsilon}_k)\frac{\nabla^{\perp}
\phi(X^{\epsilon}_k)}{|\nabla\phi(X^{\epsilon}_k)|}, \frac{\nabla^{\perp}\phi(X^{\epsilon}_k)}
{|\nabla\phi(X^{\epsilon}_k)|}\right> +o(\epsilon^2)\\
=& \epsilon^2\left(\partial_t\phi(X^{\epsilon}_k) + |\nabla \phi(X^{\epsilon}_k)|
(\beta_kf(s_k)-\frac{b_k^2s_k^2}{2}\curv(\phi)_{|X^{\epsilon}_k})\right)+o(\epsilon^2)
\end{aligned}
\end{equation}
in view of
\beq
-\curv(\phi) |D\phi|=|\nabla \phi| \dive \left(\frac{\nabla \phi}{|\nabla \phi|}\right) = 
\left<D^2\phi\frac{\nabla^{\perp}\phi}{|\nabla\phi|}, \frac{\nabla^{\perp}\phi}{|\nabla\phi|}\right>.
\eeq
We use the notation
\beq
\varphi_s(\kappa) = \kappa \frac{s^2}{2}- f(s)
\eeq
for $\kappa \in \mathbb{R}$ and $s>0$ 
and we get
\begin{equation} \label{85}
\begin{aligned}
\phi(X^{\epsilon}_{k+1})-\phi(X^{\epsilon}_k) &\le \epsilon^2
\left\{\partial_t\phi(X^{\epsilon}_k) -
|\nabla \phi(X^{\epsilon}_k)|\varphi_{s_k}(\curv(\phi)_{|X^{\epsilon}_k})\right\}+o(\epsilon^2) 
\end{aligned}
\end{equation}
We want to show by using (\ref{20}) that  
\beq \label{a2}
\partial_t\phi(X^{\epsilon}_k) -
|\nabla \phi(X^{\epsilon}_k)|\varphi_{s_k}(\curv(\phi)_{|X^{\epsilon}_k}) \le \frac{\theta_0}{2}
\eeq
for small $\epsilon$. 
Let us assume for a moment that this is already shown then we have
\begin{equation}\label{93}
\begin{aligned}
\phi(X^{\epsilon}_{k+1})-\phi(X^{\epsilon}_k) 
&\le \epsilon^2 \frac{\theta_0}{2}
\end{aligned}
\end{equation}
and hence
\beq \label{22}
\phi(X^{\epsilon}_k)-\phi(X^{\epsilon}_0) \le \frac{k}{2}\epsilon^2\theta_0<0. 
\eeq
Adding this inequality to (\ref{21}) gives
\beq \label{26}
u^{\epsilon}(X^{\epsilon}_0)-\phi(X^{\epsilon}_0) \le u^{\epsilon}(X^{\epsilon}_k)-\phi(X^{\epsilon}_k)
+\frac{k}{2}\epsilon^2\theta_0.
\eeq
For each $\epsilon$ we consider the finite sequence $(X^{\epsilon}_k)_{0\le k\le k(\epsilon)}$ where 
$k(\epsilon)$ is as follows. Let  
\beq
U =  B_{\de}(x_0, t_0)\backslash B_{\de_1}(x_0, t_0), \quad 0<\de_1<\de,
\eeq
we may assume that $X^{\epsilon}_0 \in  B_{\de_1}(x_0, t_0)$ and that
\beq \label{a1}
u^{\epsilon}(X^{\epsilon}_0)-\phi(X^{\epsilon}_0) \ge \bar u(x_0, t_0)-\phi(x_0, t_0)-\frac{\theta_0}{4}.
\eeq
Let
\beq
k(\epsilon) =\min\{k \in \mathbb{N}: X^{\epsilon}_k \in U\}.
\eeq
Note, that $k(\epsilon)$ is well-defined since (\ref{26}) and (\ref{a1}) imply that the sequence $(X^{\epsilon}_k)$
leaves $B_{\de_1}(x_0, t_0)$ at the latest for $k=[\epsilon^{-2}]$. Hence (for small $\epsilon$ compared to $\de$) there must be an element $X^{\epsilon}_k\in U$.

For a subsequence $X^{\epsilon}_{k(\epsilon)}\rightarrow (x', t') \neq (x_0, t_0)$ and by (\ref{26}) we get
\beq
(\bar u-\phi)(x_0, t_0) \le (\bar u-\phi)(x', t')
\eeq
in contradiction to the fact that $(x_0, t_0)$ is the unique maximum in $B_{\de}(x_0, t_0)$.

We show (\ref{a2}) by distinguishing cases, thereby we assume $\de$ to be sufficiently small. We set $\kappa_0=\curv(\phi)_{|X_0}$ and 
$\kappa_k = \curv(\phi)_{|X^{\epsilon}_k}$.

{\it Case (1):}
If $\kappa_0<0$ then there exists $\epsilon_0=\epsilon_0(\kappa_0)>0$ so that 
we can choose $s_k \in I_{\epsilon}$ with
\beq
\varphi_{s_k}(\kappa_k) = \varphi(\kappa_k)
\eeq
for all $0<\epsilon <\epsilon_0$. Then we use (\ref{20}).

{\it Case (2):} We assume $\kappa_0> 0$. For small $\epsilon$
we have
\beq
\varphi_{\epsilon^{\al_1}}(\kappa_0) < 0 < \varphi_{\epsilon^{-\al_2}}(\kappa_0)
\eeq
so that there is a $s_k \in I_{\epsilon}$ (actually not depending on $k$) with
\beq
\varphi_{s_k}(\kappa_0) =0.
\eeq
It follows that (for $\de \rightarrow 0$)
\beq
\varphi_{s_k}(\kappa_k)\rightarrow 0 = \varphi(\kappa_0).
\eeq

{\it Case (3):}
We assume $\kappa_0=0$. 

{\it Case (3a):} Let $\kappa_k\ge 0$.
For small $\epsilon$ we have 
\beq 
\varphi_{\epsilon^{\al_1}}(\kappa_k)<0.
\eeq
If 
\beq
\varphi_{\epsilon^{-\al_2}}(\kappa_k)>0
\eeq
we choose $s_k \in I_{\epsilon}$ with
\beq
\varphi_{s_k}(\kappa_k)=0
\eeq
and are ready, otherwise we set $s_k =\epsilon^{-\al_2}$ and get (as $\epsilon \rightarrow 0$)
\beq
0 \leftarrow -f(\epsilon^{-\al_2})\le \varphi_{s_k}(\kappa_k)\le 0
\eeq
which also implies the claim.

{\it Case (3b):} Let $\kappa_k< 0$.
Let 
\beq \label{101}
\frac{2\al_2}{1-\ga}>\mu>2 \al_2
\eeq
and distinguish the following cases.

{\it Case (3ba):}
If $|\kappa_k|\le \epsilon^{\mu}$ we set $s_k=\epsilon^{-\al_2}$
and can estimate
\beq
-\varphi_{s_k}(\kappa_k) \le \epsilon^{\mu-2\al_2}+c\epsilon^{\al_2
\frac{2\ga}{1-\ga}}.
\eeq

{\it Case (3bb):}
Let $|\kappa_k|> \epsilon^{\mu}$. We calculate the maximizing $s>0$ of 
$\varphi_s(\kappa_k)$. 
For general $\kappa<0$ the maximizing $s>0$ of $\varphi_s(\kappa)$ is given by
\beq \label{106}
\frac{d}{ds}\varphi_{s}(\kappa) =0
\eeq
hence
\begin{equation} \label{105}
\begin{aligned}
 s_{\max} =& \left\{ \frac{\ga-1}{2\ga c_{\ga}}\kappa\right\}^{\frac{\ga-1}{2}}.
 \end{aligned}
\end{equation} 
Evaluated for $\kappa = \kappa_k$ we get the maximizing 
\begin{equation}
\begin{aligned}
 s_k =& \left\{ \frac{\ga-1}{2\ga c_{\ga}} \kappa_k\right\}^{\frac{\ga-1}{2}} \\
\le& c\epsilon^{\mu\frac{\ga-1}{2}}.
 \end{aligned}
\end{equation} 
The so defined $s_k$ lies in $I_{\epsilon}$ if 
\beq
\mu \frac{\ga-1}{2}>-\al_2
\eeq 
which is the case in view of (\ref{101}).

(ii)  Let us consider the case $\nabla \phi(x_0, t_0)= 0$. We define the $X^{\epsilon}_k$ as 
before by (\ref{25}) but we replace in equation (\ref{25}) the directions 
$\frac{\nabla^{\perp}\phi(X^{\epsilon}_k)}{|\nabla \phi(X^{\epsilon}_k|}$ and 
$\frac{\nabla \phi(X^{\epsilon}_k)}{|\nabla \phi(X^{\epsilon}_k|}$ by
$v_k^{\perp}$ and $v_k$ respectively, $v_k$ an unit vector with
\beq
\left<v_k^{\perp}, \nabla \phi(X^{\epsilon}_k)\right>=0.
\eeq
The $s_k$ are chosen as will become clear in the following and the $\beta_k$ and $b_k$ are chosen as before, i.e. so that they maximize $u^{\epsilon}(X^{\epsilon}_{k+1})$ for given $s_k, v_k$.
Then we have $u^{\epsilon}(X^{\epsilon}_0) \le u^{\epsilon}(X^{\epsilon}_k)$ as before. 
Looking at (\ref{260}) and (\ref{110}) and adapting it to the present
situation gives
\begin{equation} \label{261}
\begin{aligned}
\phi(X^{\epsilon}_{k+1}) -& \phi(X^{\epsilon}_k)  \\
=& \epsilon^2\left(\partial_t\phi(X^{\epsilon}_k)+ \left<\nabla\phi
(X^{\epsilon}_k), v_k\right>\beta_k f(s_k) \right) \\
&+ \frac{\epsilon^2}{2}b_k^2s_k^2\left<D^2\phi(X^{\epsilon}_k)v_k^{\perp}, v_k^{\perp}\right> +o(\epsilon^2).
\end{aligned}
\end{equation}
Our goal is to ensure that (\ref{93}) also holds in this case.
We set  
\beq
\la^{\epsilon}_k:=|\left<\nabla \phi(X^{\epsilon}_k), v_k\right>|
\rightarrow 0
\eeq
as $\de \rightarrow 0$.
If $\la^{\epsilon}_k=0$ then we are ready by setting 
$s_k = \epsilon^{\al_1}$. Let us assume that $\la^{\epsilon}_k\neq 0$.
We get the following inequality 
\begin{equation} \label{88}
\begin{aligned}
\phi(X^{\epsilon}_{k+1})-\phi(X^{\epsilon}_k) &\le \epsilon^2
\left\{\partial_t\phi(X^{\epsilon}_k) -\la^{\epsilon}_k
\varphi_{s_k}(A^{\epsilon}_k)\right\}+o(\epsilon^2) 
\end{aligned}
\end{equation}
where
\beq
A^{\epsilon}_k = -{\la^{\epsilon}_k}^{-1}
|\left<D^2\phi(X^{\epsilon}_k)v^{\perp}_k, v^{\perp}_k\right>|.
\eeq

{\it Case (1):} 
For those $A^{\epsilon}_k$ which are contained in some sufficiently 
large interval $[a,b] \subset (-\infty, 0)$ 
(where $a,b$ do not depend on $\epsilon$ and $k$) we 
choose $s_k\in I_{\epsilon}$
so that 
$\varphi_{s_k}( A^{\epsilon}_k ) =  \varphi( A^{\epsilon}_k)$ 
for small
$\epsilon$.

{\it Case (2):} 
If $A^{\epsilon}_k\ge 0$ we set $s_k=\epsilon^{ -\al_2}$ and are ready.

{\it Case (3):} 
If $0>A^{\epsilon}_k >b$ we choose $\mu$ according to (\ref{101})
and distinguish cases.

{\it Case (3a):} If 
$| \la^{\epsilon}_kA^{\epsilon}_k|\le \epsilon^{\mu}$ then we set $s_k = \epsilon^{-\al_2}$ 
and have
\beq
|\la^{\epsilon}_k \varphi_{s_k}(A^{\epsilon}_k)| \le c\epsilon^{\mu-2\al_2}+c 
\epsilon^{\frac{2\al_2 \ga}{1-\ga}}.
\eeq

{\it Case (3b):}
If
$| \la^{\epsilon}_kA^{\epsilon}_k|> \epsilon^{\mu}$ 
then $|A^{\epsilon}_k|> \frac{\epsilon^{\mu}}{\la^{\epsilon}_k}$.
We calculate the maximizing $s>0$ of $\varphi_s(A^{\epsilon}_k)$ according to
(\ref{105}) and get
\begin{equation}
\begin{aligned}
 s_k =& \left\{ \frac{\ga-1}{2\ga c_{\ga}}A^{\epsilon}_k\right\}^{\frac{\ga-1}{2}} \\
\le& c\epsilon^{\mu\frac{\ga-1}{2}}{\la^{\epsilon}_k}^{\frac{1-\ga}{2}}.
 \end{aligned}
\end{equation} 
The so defined $s_k$ lies in $I_{\epsilon}$ if 
\beq
\mu \frac{\ga-1}{2}>-\al_2
\eeq 
which is the case in view of (\ref{101}).

{\it Case (4):} 
If $A^{\epsilon}_k < a$ we choose 
\beq \label{95}
\frac{2\al_1}{1-\ga}>\mu > \ga \frac{2\al_1}{1-\ga}
\eeq 
and distinguish cases.

{\it Case (4a):}
We assume $\la^{\epsilon}_k\le \epsilon^{\mu}$.
We set $s_k = \epsilon^{\al_1}$, calculate
\beq
\la^{\epsilon}_k f(s_k) \le c\epsilon^{\mu+\al_1\frac{2\ga}{\ga-1}}\rightarrow 0
\eeq
so that (\ref{93}) follows as well. 

{\it Case (4b):}
We assume $\la^{\epsilon}_k> \epsilon^{\mu}$. 
Then 
\beq
|A^{\epsilon}_k| \le c \epsilon^{-\mu}.
\eeq
We calculate the maximizing $s>0$ of $\varphi_s(A^{\epsilon}_k)$
and get
\begin{equation}
\begin{aligned}
 s_k =& \left\{ \frac{\ga-1}{2\ga c_{\ga}}A^{\epsilon}_k\right\}^{\frac{\ga-1}{2}} \\
\ge& c\epsilon^{\mu\frac{1-\ga}{2}}.
 \end{aligned}
\end{equation} 
To ensure that the so defined $s_k$ lies in $I_{\epsilon}$ we need 
\beq
\al_1>\mu \frac{1-\ga}{2} 
\eeq 
which follows from (\ref{95}). There holds
\beq
\varphi_{s_k}(A^{\epsilon}_k) = 
\varphi(A^{\epsilon}_k) = -|A^{\epsilon}_k|^{\ga}
\eeq
and hence 
\beq
\la^{\epsilon}_k\varphi_{s_k}(A^{\epsilon}_k) = -{\la^{\epsilon}_k}^{1-\ga}{|\left<D^2\phi_t(
X^{\epsilon}_k)v_k^{\perp}, v_k^{\perp}\right>|}^{\ga} \rightarrow 0.
\eeq 

In all cases we get
\beq
\phi(X^{\epsilon}_{k+1})-\phi(X^{\epsilon}_k) \le \epsilon^2 \frac{\theta_0}{2}.
\eeq
Now, we argument as in (i). 
\end{proof}

\begin{lemma} \label{d1}
$\underline u$ is a viscosity supersolution of (\ref{3}).
\end{lemma}
\begin{proof}
We argue by contradiction. If not, then there is a smooth $\phi$ such that $(x_0, t_0)$ is a local
minimum of $\underline{u}-\phi$ with
 \beq \label{30}
 \phi_t -|\nabla \phi|\varphi(\curv(\phi)) \ge \theta_0 > 0 \text{ at }(x_0, t_0)
 \eeq
 if $\nabla \phi(x_0, t_0)\neq 0$ or we have $\nabla \phi(x_0, t_0)=0$ and 
 \beq
 \phi_t \ge \theta_0 >0 
 \eeq
 in $(x_0, t_0)$. 
 Again without loss of generality, we can assume the minimum is strict. 
 Changing $\theta_0$ if necessary, we can also find a $\de$-neighborhood of $(x_0, t_0)$ 
 in which these assertions hold. Again we can find 
 $(x^0_{\epsilon}, t^0_{\epsilon})\rightarrow (x_0, t_0)$ such that 
 $ u^{\epsilon}(x^0_{\epsilon}, t^0_{\epsilon})\rightarrow \underline u(x_0, t_0)$. 
 Taking $s_0$ and the unit-norm $v_0$ that achieve the minimum in the characterization (\ref{40}) we find
 \beq
 u^{\epsilon}(x^0_{\epsilon}, t^0_{\epsilon}) = \max_{b, \beta\in\{-1,1\}}
 u^{\epsilon}(x^{\epsilon}_0+b \epsilon s_0 v_0^{\perp} + \beta \epsilon^2 f(s_0)v_0, t^{\epsilon}_0 + \epsilon ^2).
 \eeq
 We set $X^{\epsilon}_0 = (x^0_{\epsilon}, t^0_{\epsilon})$, 
 $X^{\epsilon}_1 = (x+b_0 \epsilon s_0 v_0^{\perp} + \beta \epsilon^2 f(s_0)v_0, t^{\epsilon}_0 + \epsilon ^2)$
and inductively
\beq \label{210}
X^{\epsilon}_{k+1}= X^{\epsilon}_k + 
(b_k \epsilon s_k v_k^{\perp} + \beta_k \epsilon^2 f(s_k)v_k, \epsilon ^2)
\eeq
where $s_k, v_k$ are chosen recursively (and also depend on $\epsilon) $ so that the 
minimum is attained and $b_k, \beta_k$ will be chosen later. We have for all $b_k, \beta_k\in \{1,-1\}$
\beq
u^{\epsilon}(X^{\epsilon}_k) \ge u^{\epsilon}(X^{\epsilon}_{k+1})
\eeq
and thus 
\beq
u^{\epsilon}(X_0) \ge u^{\epsilon}(X_k).
\eeq
On the other hand, extending the $X^{\epsilon}_k$ into an affine path by affine interpolation
as before
and doing a Taylor expansion gives -- adapting (\ref{260}) and (\ref{110}) to the present
situation -- 
\begin{equation} \label{402}
\begin{aligned}
\phi(X^{\epsilon}_ {k+1})&-\phi(X^{\epsilon}_k) 
=  \epsilon b_k s_k \left<\nabla \phi (X^{\epsilon}_k), v_k^{\perp}\right> \\
& + \epsilon^2\left(\partial_t \phi(X^{\epsilon}_k)
+\beta_k f(s_k) \left<\nabla \phi(X^{\epsilon}_k), v_k\right>+
\frac{b_k^2s_k^2}{2}\left<D^2\phi(X^{\epsilon}_k)v_k^{\perp}, v_k^{\perp}\right>\right)  \\
& + O(\epsilon^3).
\end{aligned}
\end{equation}
Our goal is to show 
\begin{equation}\label{108}
\begin{aligned}
\phi(X^{\epsilon}_{k+1})-\phi(X^{\epsilon}_k) 
&\ge \epsilon^2 \frac{\theta_0}{2}
\end{aligned}
\end{equation}
for all sufficiently small $\de>0$.

(i) 
Let us consider the case where $\nabla \phi(x_0, t_0)= 0$. Let us denote the RHS of (\ref{210})
for $b_k=\beta_k=1$ by $z_1$ and for $b_k=\beta_k=-1$ by $z_2$. 
By replacing $v_k$, $s_k$, $b_k$ and $\beta_k$ by suitable 
$\tilde v_k$, $\tilde s_k$, $\tilde b_k$ and $\tilde \beta_k$ respectively, we may assume that
\beq \label{1006}
z_1= X^{\epsilon}_k + 
(\tilde b_k \epsilon \tilde s_k \tilde v_k^{\perp} + \tilde \beta_k \epsilon^2 
f(\tilde s_k) \tilde v_k, \epsilon^2)
\eeq
with $|\tilde s_k|\le \epsilon^{\mu}$ where $0<\mu<1$. 
We note that if we change the signs of $\tilde b_k$ and $\tilde \beta_k$ on the RHS of (\ref{1006}) then this 
RHS is equal to $z_2$. We explain why (\ref{1006}) holds.
 Let us denote for $0<a<b$ the following piece of 
a graph by
\beq
\Gamma(a,b) = \left\{(\epsilon s, \epsilon^2f(s)): a\le s \le b\right\}.
\eeq
The minimum of $\dist(0, \cdot)$ on $\Gamma=\Gamma(\epsilon^{\al_1}, \epsilon^{-\al_2})$
is attained in $(\epsilon \tilde s, \epsilon^2f(\tilde s))$ where  
\beq \label{1005}
\tilde s = \left\{\frac{1-\ga}{2\ga \epsilon^2}\right\}^{\frac{\ga-1}{2\ga+2}},
\eeq
here and in the following we denote by the origin 0 the point $X^{\epsilon}_k$.
Furthermore, the distance of the left endpoint
of $\Gamma$ to 
the origin is larger than 
the distance of the right endpoint of $\Gamma$
to the origin, to see this we omit constants and check the orders of $\epsilon$ in the squares
of these distances. We get 
\begin{equation} \label{left_endpoint}
\begin{aligned}
\epsilon^{2+2\al_1} + \epsilon^{4-\al_1\frac{4\ga}{1-\ga}}
\end{aligned}  
\end{equation}
for the left endpoint and
\begin{equation}
\begin{aligned}
\epsilon^{2-2\al_2} + \epsilon^{4+\al_2\frac{4\ga}{1-\ga}}
\end{aligned}  
\end{equation}
for the right endpoint.
Therefore we must have
\beq
4- \al_1 \frac{4\ga}{1-\ga} < 2-2\al_2
\eeq
which follows from (\ref{401}).
Hence all points $p \in \Gamma(\tilde s, \epsilon^{-\al_2})$
can be 'covered' by 
rotating $\Gamma(\epsilon^{\al_1}, \tilde s)$
around the origin, i.e. 
there is a rotation $O_p$ of
the plane around the origin (depending on $p$) such that
\beq \label{7001}
p \in O_p\Gamma(\epsilon^{\al_1}, \tilde s).
\eeq
 But this rotation can be realized by choosing the quantities with a tilde 
suitable.

Now, we look at the RHS of (\ref{402}) in the situation that $s_k, v_k, b_k, \beta_k$ carry a tilde.
The only possibly 'bad' summands therein are the summand which contains $\tilde b_k$ and the one
with $\tilde \beta_k$ (note, that $\tilde s_k$ is small). Let us denote the first by $A$ and the second by $B$.
If $A+B\ge0$ we set
 $X^{\epsilon}_{k+1}=z_1$ and if $A+B<0$ we set  $X^{\epsilon}_{k+1}=z_2$.

(ii)
Let us consider the case where $\nabla \phi(x_0, t_0)\neq 0$ and set 
$d_0=|\nabla \phi(x_0, t_0)|$. We let $p$, $-p$, $q$ and $-q$ be the four possible points for $X^{\epsilon}_{k+1}$ depending on the choice of $b_k, \beta_k$ where we let $p$ be the point with $b_k=\be_k=1$. We choose $b_k, \beta_k$ so that the summands in which they appear are non-negative. Let $c_0>\tilde s$ 
be a sufficiently small fixed constant (not depending on $\epsilon$) then (\ref{108}) follows from (\ref{402}) as can be seen from the following cases.

{\it Case (1):} We assume $s_k \le c_0$. 
Let us denote the summands of the resulting RHS of (\ref{402})
from left to right by $A$, $B$, $C$, and $D$. If $|\left<\nabla \varphi(X^{\epsilon}_k), v_k^{\perp}\right>|>\frac{d_0}{2}$ then $A\ge \frac{d_0}{2}\epsilon^{1+\al_1}\ge \theta_0\epsilon^2$ is the dominating summand, otherwise 
$ |\left< \nabla \varphi (X^{\epsilon}_k), v_k \right> | > \frac{d_0}{2}$ so that $C\ge \theta_0\epsilon^2 $ dominates $B$ and $D$ and hence the claim follows. 

{\it Case (2):} We assume $s_k > c_0$.

{\it Case (2a):} Assume that there is a sufficiently large constant $c_1>0$ so that
\beq
\left<\nabla \phi(X^{\epsilon}_k), v_k^{\perp}\right> > c_1 \epsilon s_k.
\eeq
Then the first summand on the RHS of (\ref{402}) is dominating and gives the desired estimate.

{\it Case (2b):} Let 
\beq
\left<\nabla \phi(X^{\epsilon}_k), v_k^{\perp}\right> \le c_1 \epsilon s_k.
\eeq

{\it Case (2ba):} Let $\curv(\phi)_{|X^{\epsilon}_k}\le 0$.

We have
\beq
v_k^{\perp} = \frac{\nabla \phi(X^{\epsilon}_k)^{\perp}}{|\nabla \phi(X^{\epsilon}_k)|}+c \epsilon s_k
\eeq
where we possibly have to change the signs of both $b_k$ and $\beta_k$.
We denote the summands of the RHS of (\ref{402}) from left to right by $A$, $B$, $C$ and $D$. We have
\beq
|A| \le c \epsilon^2 s_k^2.
\eeq
and
\begin{equation}
\begin{aligned}
\epsilon^{-2}&(B+C+D) \\
\ge & \partial_t\phi(X^{\epsilon}_k) + f(s_k) |\nabla \phi(X^{\epsilon}_k)| \\
& + \frac{s_k^2}{2}\left<D^2\phi(X^{\epsilon}_k)\frac{\nabla \phi(X^{\epsilon}_k)^{\perp}}{|\nabla \phi(X^{\epsilon}_k)|}, \frac{\nabla \phi(X^{\epsilon}_k)^{\perp}}{|\nabla \phi(X^{\epsilon}_k)|}\right> -c\epsilon s_k^3  \\
= & \partial_t\phi(X^{\epsilon}_k) + f(s_k) |\nabla \phi(X^{\epsilon}_k)| -   |\nabla \phi(X^{\epsilon}_k)| \frac{s_k^2}{2}\curv (\phi)_{|X^{\epsilon}_k}
-c\epsilon s_k^3  \\
= &  \partial_t\phi(X^{\epsilon}_k) -  |\nabla \phi(X^{\epsilon}_k)|\varphi_{s_k}(\curv (\phi)_{|X^{\epsilon}_k})
-c\epsilon s_k^3  \\
\ge &  \partial_t\phi(X^{\epsilon}_k) -  |\nabla \phi(X^{\epsilon}_k)|\varphi(\curv (\phi)_{|X^{\epsilon}_k})
-c\epsilon s_k^3  \\
\ge &  \frac{\Theta_0}{2}
-c\epsilon s_k^3 
\end{aligned}
\end{equation}
Now, the claim follows from assumption (\ref{401}).

{\it Case (2bb):} Let $\curv(\phi)_{|X^{\epsilon}_k}> 0$. W.l.o.g. we may assume that
\beq
\curv(\phi)_{|(x_0, t_0)}\ge 0,
\eeq
 otherwise we choose a smaller $\de$-neighborhood.
Then $\partial_t \phi(X^{\epsilon}_k)\ge \frac{\Theta}{2}$.
We define 
\beq \label{1010}
\tilde v_k = O_p v_k, \quad \tilde v_k^{\perp} = O_p v_k^{\perp}
\eeq
where $O_p$ as in (\ref{7001}) (and $p$ the point corresponding to $b_k=\beta_k=1$)
then
\beq
p = X^{\epsilon}_k + 
( \epsilon \tilde s_k \tilde  v_k^{\perp} + \tilde  \epsilon^2 f(\tilde s_k)\tilde  v_k, 
\epsilon ^2)
\eeq
with suitable small $\tilde s_k$. Then
\beq
\left|\left<\nabla \phi(X^{\epsilon}_k), \tilde v_k^{\perp}\right>\right| \ge \tilde c_0>0.
\eeq
We denote the terms on the RHS of (\ref{402}) (now in the situation of the quantities with a tilde) from left to right by $A$, $B$, $C$ and $D$. There holds $\frac{B}{2}\ge |D|$. By considering $-p$ instead of $p$ we may assume that $A+C\ge 0$. Since $B \ge \epsilon^2 \frac{\Theta}{2}$ the claim follows.

We used that the rotation angle $\al$ of $O_p$ is almost $\frac{\pi}{2}$, hence we supplement an estimate
for $\al$. We write $\al=\al_2-\al_1$ with
\beq
\tan  \al_1 = \frac{\epsilon^2f(s_k)}{\epsilon s_k} = c\epsilon s_k^{\frac{1+\ga}{\ga-1}}, \quad s_k \ge c_0
\eeq
and
\beq
\tan \al_2 = \frac{\epsilon^2f(\tilde s_k)}{\epsilon \tilde s_k}, \quad \tilde s_k \le \tilde s
\eeq
where 
\beq
\epsilon^2 f( \tilde s_k) \ge c \epsilon s_k,
\eeq
or, equivalently,
\beq
\tilde s_k \le c\epsilon^{\frac{1-\ga}{2\ga}}s_k^{\frac{\ga-1}{2\ga}}
\eeq
so that 
\beq
\tan \al_2 \ge \epsilon^{\frac{\ga-1}{2\ga}}s_k^{\frac{1+\ga}{2\ga}}.
\eeq
Altogether this gives
\beq
\al_1 = O(\epsilon s_k^{\frac{1+\ga}{\ga-1}}) \quad \wedge \quad \left|\al_2-\frac{\pi}{2}\right| = 
O\left(\epsilon^{\frac{1-\ga}{2\ga}}s_k^{-\frac{1+\ga}{2\ga}}\right).
\eeq
\end{proof}

\section{Numerical scheme}
\label{sec:3}

Let $\epsilon>0$ be a small step size and $h>0$ the spatial step size. We consider 
the points $\Omega_h=h\mathbb{Z}^2$
forming a rectangular grid in the plane. Our value function is given at the final time $T>0$,
\beq
u^{\epsilon}(\cdot, T) = u_0
\eeq
where we assume that $u_0=0$ in $\mathbb{R}^2\setminus \bar \Omega$.
The calculation of the value function is backward in time.
Let $x \in \Omega_h$ be a grid point. In order to calculate $u^{\epsilon}(x, t-\epsilon^2)$ from 
$u^{\epsilon}(\cdot, t)$ we
propose the following stategy. 

We discretize the control set $I_{\epsilon}$ by 
\beq
I^{\Delta s}_{\epsilon}=\{\epsilon^{\al_1}+r\Delta s: 0 \le r \le r_0\}, \quad \Delta s  = 
\frac{\epsilon^{-\al_2}-\epsilon^{\al_1}}{r_0} 
\eeq
and the set of unit vectors  $\{v\in \mathbb{R}^2:\|v\|=1\}$ by the discrete subset
\beq
S_{l_0}=\{(\cos \al, \sin \al)): \al = \frac{2\pi l}{l_0}, 0 \le lÊ\le l_0\}
\eeq
where $r_0, l_0$ are sufficiently large natural numbers. We then set
\beq \label{discrete2}
u^{\epsilon}(x, t-\epsilon^2) = \min_{v \in S_{l_0}, s \in I^{\Delta s}_{\epsilon}}\max_{b, \beta \in \{-1,1\}}u^{\epsilon}
(z, t)
\eeq
where 
\beq \label{ab1}
z= x + b \epsilon s v 
+ \beta \epsilon^2 f(s) v^{\perp}
\eeq
and
$u^{\epsilon}$ on the RHS of (\ref{discrete2}) is replaced by the value of a bilinear 
interpolating function if $z$ is not a grid point.

Let us look at the relative scaling of $\epsilon$ and $h$ in the special case of $\ga=0.5$. 
From (\ref{1005}) and (\ref{left_endpoint}) we know that the distance from $z$ to $x$ is at least 
$c\epsilon^{\frac{4}{3}}$ and at most $c\epsilon^{2(1-\al_1)}$, note, that $\frac{1}{2}<\al_1<1$ and 
$0<\al_2<\min(2\al_1-1, \frac{1}{3})$. In order to include second order information 
of the value function in the optimization problem (\ref{discrete2}) it is natural to 
assume that the range of the control variable $z$
is not covered by a square of size $h^2$, so that we assume $h<\epsilon^{\frac{4}{3}}$. Then we get for the time step size the estimate $\epsilon^2>h^{\frac{3}{2}}$ and the standard interpolation error in the supremum norm is of order $h^2=o(\epsilon^2)$.

\section{Numerical examples}

\label{sec:4}

We consider the case of a shrinking circle with initial radius 
$R_0=1$ and different values of $\frac{1}{3}<\gamma<1$. As initial function we use
\beq
u_0(x)=\max(|x|^{\ga+1}-R_0^{\ga+1},0)^2
\eeq
with solution 
\beq
u(x,t) = \max(|x|^{\ga+1}-R_0^{\ga+1}+(\ga+1)t, 0)^2
\eeq
for $0 \le t < T_{\max}=\frac{R_0^{\gamma+1}}{\gamma+1}$.
To define the value function $u^{\epsilon}$ completely we have to specify the parameters
$\epsilon, \alpha_1$ and $\alpha_2$.
For the numerical approximation of $u^{\epsilon}$ we have to specify the discretization 
parameters $h, l_0, r_0$.

\subsection{Search of good parameters $\al_1, \al_2$ for the example case 
$\epsilon=0.08$ and $\gamma=0.7$}
 \label{subsec:1}
 In this subsection we consider the case $\epsilon=0.08$ and $\gamma=0.7$ (then 
 $T_{\max}\approx 0.59$) and 
 test different values of $\al_1, \al_2$ while the 
 numerical parameters $h=0.01$, $l_0=160$ are fixed and 
$r_0$ is adapted so that $I_{\epsilon}$ is discretized equidistantly 
with step size $0.01$.

In order to determine good values for $\alpha_1$ and $\alpha_2$ we 
denote the lower bound for $\alpha_1$ in 
(\ref{601})
by $m_1$, the upper bound by $m_2$ and the upper bound for 
$\alpha_2$ in (\ref{401}) by $m_3$. We introduce a scaling factor 
$\scale$ and set
\beq
\alpha_1 = m_1+\scale \cdot (m_2-m_1) \quad \wedge \quad \alpha_2 = \scale 
\cdot m_3.
\eeq

The approximation error, i.e. the difference $u-u^{\epsilon}$, 
is measured with respect to the 
 $l_{\infty}({h\mathbb{Z}^2})$-norm and with respect to the
 $l_1(h\mathbb{Z}^2)$-norm (more precisely, we scale the latter norm by $h^2$)
 at each time step. Then the supremum over all time steps
until time $T=0.12$ of both norms is taken and denoted by $l_{\infty}$-error and $l_1$-error. 

Table \ref{tab:1} confirms the expected behavior that a larger control interval $I_{\epsilon}$ leads
to better convergence properties.

\begin{table}
\caption{Error for $\epsilon=0.08$ and $\ga=0.7$}
\label{tab:1}       
\begin{tabular}{llll}
\hline\noalign{\smallskip}
 & $\scale=0.1$ & $\scale=0.3$ & $\scale=0.9$  \\
\noalign{\smallskip}\hline\noalign{\smallskip}
$l_{\infty}$-error & 0.1085 & 0.08432 & 0.08431\\
$l_1$-error & 0.1571 & 0.1006 & 0.0988\\
\noalign{\smallskip}\hline
\end{tabular}
\end{table} 

\subsection{Influence of the numerical parameters $h$, $l_0$ and $r_0$}
We consider for the case $\epsilon=0.08$, $\gamma=0.7$ and $\scale=0.9$
three scenarios in which we analyze the influence of the numerical parameters $h$, $r_0$ and $l_0$ on
the accuracy of our algorithm. Therefore we keep hold two of these three numerical parameters in each scenario  
while the 
remaining one runs through some test values. The corresponding errors are presented in Tables \ref{table2},
\ref{table3} and \ref{table4}.

\begin{table}
\caption{Influence of $h$  on the error for $\epsilon=0.08$, $\gamma=0.7$, 
$\scale=0.9$, $l_0=160$ and $r_0=160$}
\label{table2}       
\begin{tabular}{lll}
\hline\noalign{\smallskip}
& $l_{\infty}$-error & $l_1$-error \\ 
\noalign{\smallskip}\hline\noalign{\smallskip}
$h=0.16$ & 0.2639 & 0.1876\\
$h=0.08$ & 0.1180 & 0.0999\\
$h=0.04$ & 0.0947 & 0.1035\\
$h=0.02$ & 0.0868 & 0.1045\\
\noalign{\smallskip}\hline
\end{tabular}
\end{table} 

\begin{table}
\caption{Influence of $r_0$  on the error for $\epsilon=0.08$, $\gamma=0.7$, 
$\scale=0.9$, $h=0.01$ and $l_0=160$}
\label{table3}       
\begin{tabular}{lll}
\hline\noalign{\smallskip}
& $l_{\infty}$-error & $l_1$-error \\ 
\noalign{\smallskip}\hline\noalign{\smallskip}
$r_0=10$ &  0.0902& 0.1075\\
$r_0=20$ &  0.0860& 0.1014\\
$r_0=40$ &  0.0848& 0.0998\\
$r_0=80$ &  0.0844& 0.0991\\
$r_0=160$ &  0.0843& 0.0988\\
\noalign{\smallskip}\hline
\end{tabular}
\end{table} 

\begin{table}
\caption{Influence of $l_0$  on the error for $\epsilon=0.08$, $\gamma=0.7$, 
$\scale=0.9$, $h=0.01$ and $r_0=160$}
\label{table4}       
\begin{tabular}{lll}
\hline\noalign{\smallskip}
& $l_{\infty}$-error & $l_1$-error \\ 
\noalign{\smallskip}\hline\noalign{\smallskip}
$l_0=10$ &  0.4254&0.4333 \\
$l_0=20$ &  0.1885& 0.2463\\
$l_0=40$ &  0.1291& 0.1327\\
$l_0=80$ &  0.0922& 0.1114\\
$l_0=160$ & 0.0843 & 0.0988\\
\noalign{\smallskip}\hline
\end{tabular}
\end{table} 

\subsection{Convergence of the value functions}
In Tables \ref{table5} and \ref{table6}
we present the approximation errors in dependence of 
$\epsilon$ for $\ga=0.8$ and $\ga=0.9$ 
(we set $\scale=0.9$).

\begin{table}
\caption{Convergence of the value functions $u^{\epsilon}$ for $\gamma=0.8$}
\label{table5}       
\begin{tabular}{llllll}
\hline\noalign{\smallskip}
& $l_{\infty}$-error & $l_1$-error&  h& $r_0$& $l_0$\\ 
\noalign{\smallskip}\hline\noalign{\smallskip}
$\epsilon=0.09$ & 0.0339 & 0.0261 & 0.01 & 100 & 360\\
$\epsilon=0.08$ & 0.0325 & 0.0253& 0.01& 100 &360\\
$\epsilon=0.05$ &0.0250  &0.0170 &0.01 & 100 &360\\
$\epsilon=0.04$ & 0.0205 & 0.0112& 0.01& 100&360\\
$\epsilon=0.02$& 0.0139&0.0130  &0.01 & 100 &360 \\
\noalign{\smallskip}\hline
\end{tabular}
\end{table}

\begin{table}
\caption{Convergence of the value functions $u^{\epsilon}$ 
for $\gamma=0.9$}
\label{table6}       
\begin{tabular}{llllll}
\hline\noalign{\smallskip}
& $l_{\infty}$-error & $l_1$-error&  h& $r_0$& $l_0$\\ 
\noalign{\smallskip}\hline\noalign{\smallskip}
$\epsilon=0.08$ & 0.1121 &0.1234 & 0.01& 80 &300\\
$\epsilon=0.04$ &  0.1060& 0.1127&0.01 & 80 &300\\
$\epsilon=0.02$ &  0.0781& 0.0756& 0.01& 80 &300\\
\noalign{\smallskip}\hline
\end{tabular}
\end{table}

\subsection{Zero level sets of the value function}
Figure \ref{fig1} shows the zero level sets of $u^{\epsilon}$ for different times. The parameters
are chosen as follows: $\ga=0.9$, $\epsilon=0.04$, $h=0.02$, 
$\scale=0.9$, $r_0=100$ and $l_0=160$.
Figure \ref{fig2} shows the same but now the initial function is replaced by
\beq
\tilde u_0(x)=\max\left((x_1^2+1.7x_2^2)^{\frac{\ga+1}{2}}-1,0
\right)^2, \quad x = (x_1, x_2)
\eeq
so that the initial curve is an ellipse.
One can see that the evolving curves, i.e. the zero level sets 
of $u^{\epsilon}$, become circular and shrink to a point which is the known behavior for the 
level sets of the limit function $u$.

\begin{figure}
 \includegraphics[width=3.5cm]{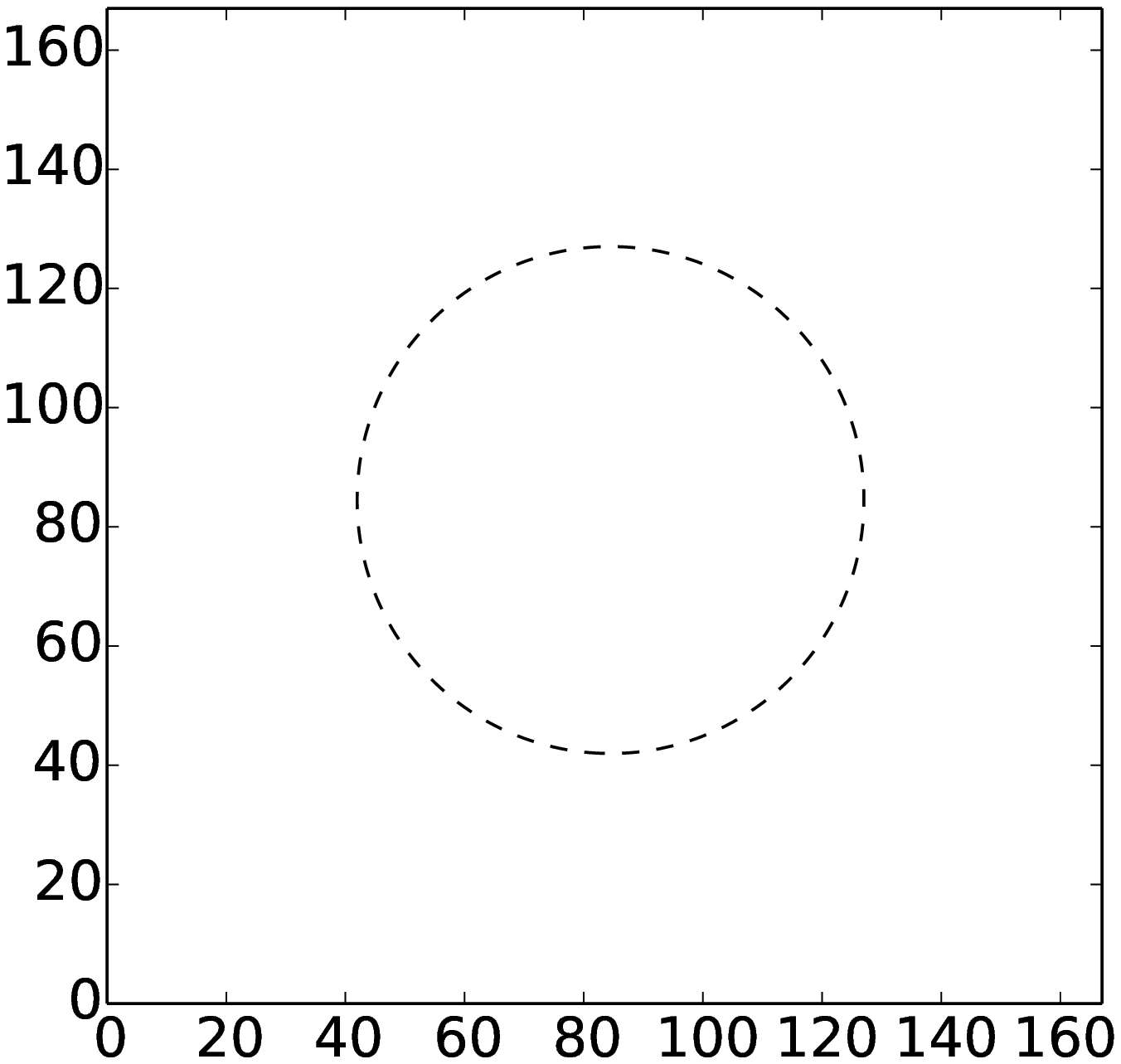} 
 \includegraphics[width=3.5cm]{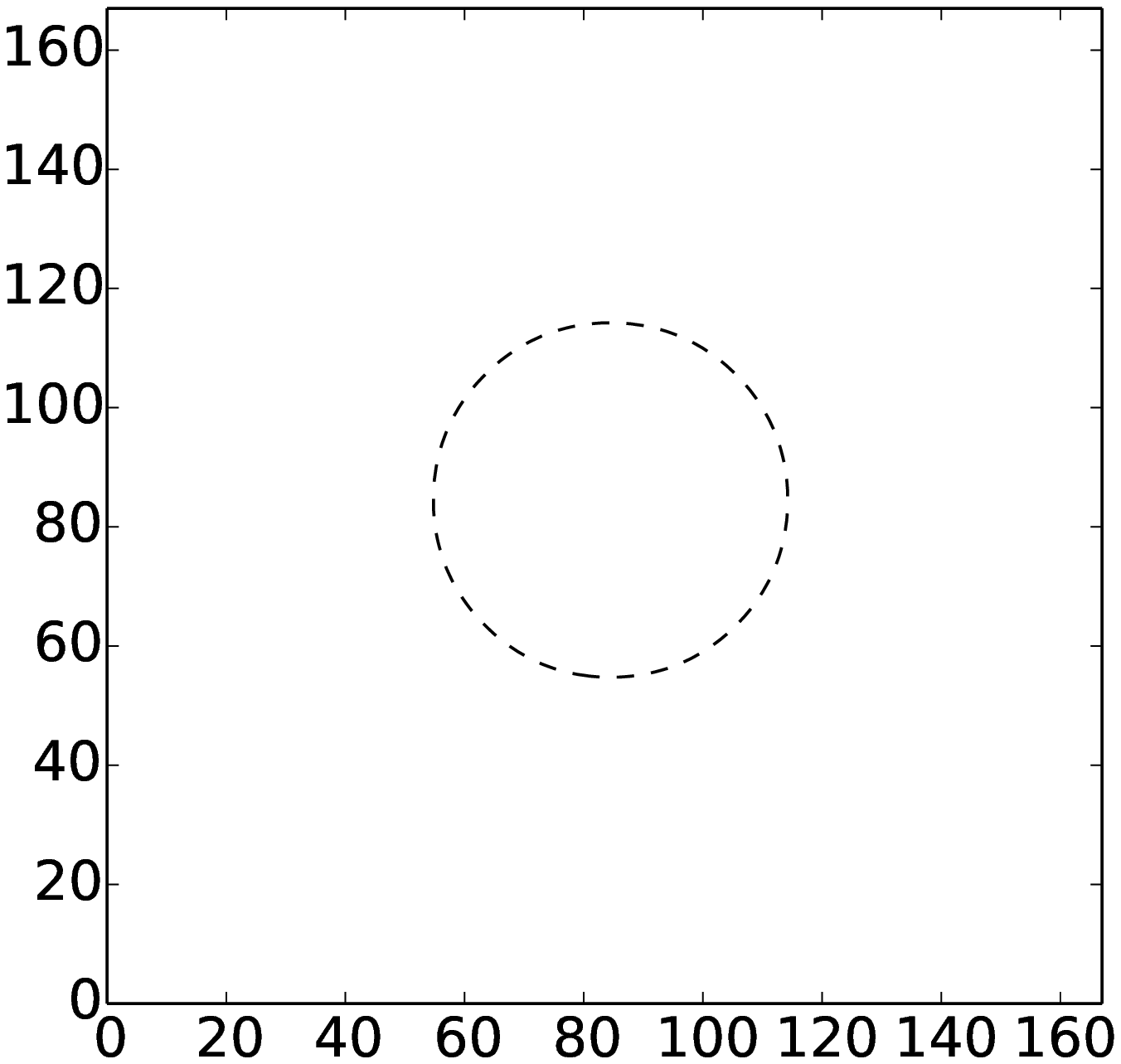}
 \includegraphics[width=3.5cm]{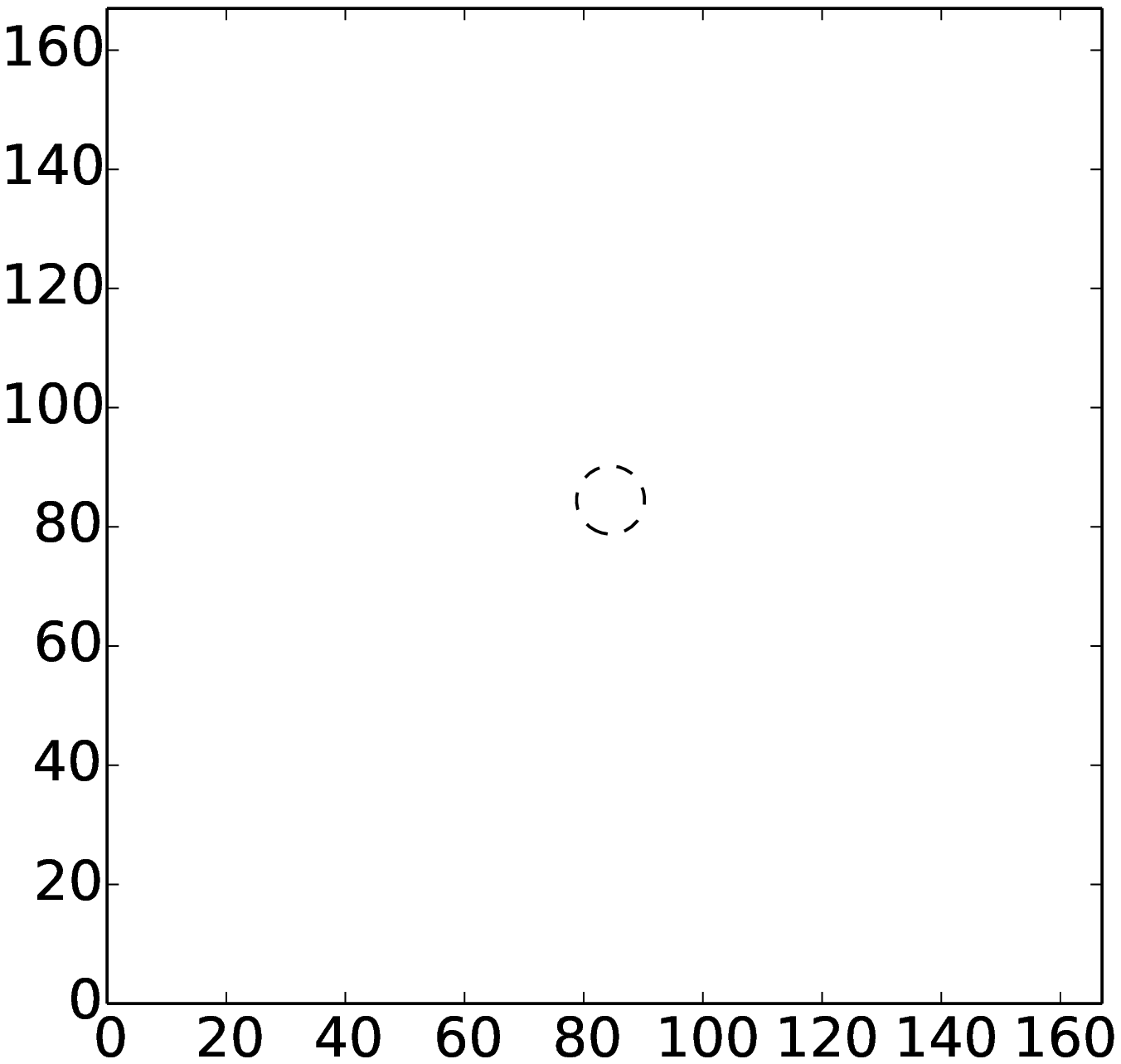}
 \caption{Circle-case: zero-level set (numerically -0.07-level) of $u^{\epsilon}$ with initial function $u_0$
 for $T=0$, $T=0.24$ and $T=0.48$.}
 \label{fig1}
 \end{figure}
 
\begin{figure}
 \includegraphics[width=3.5cm]{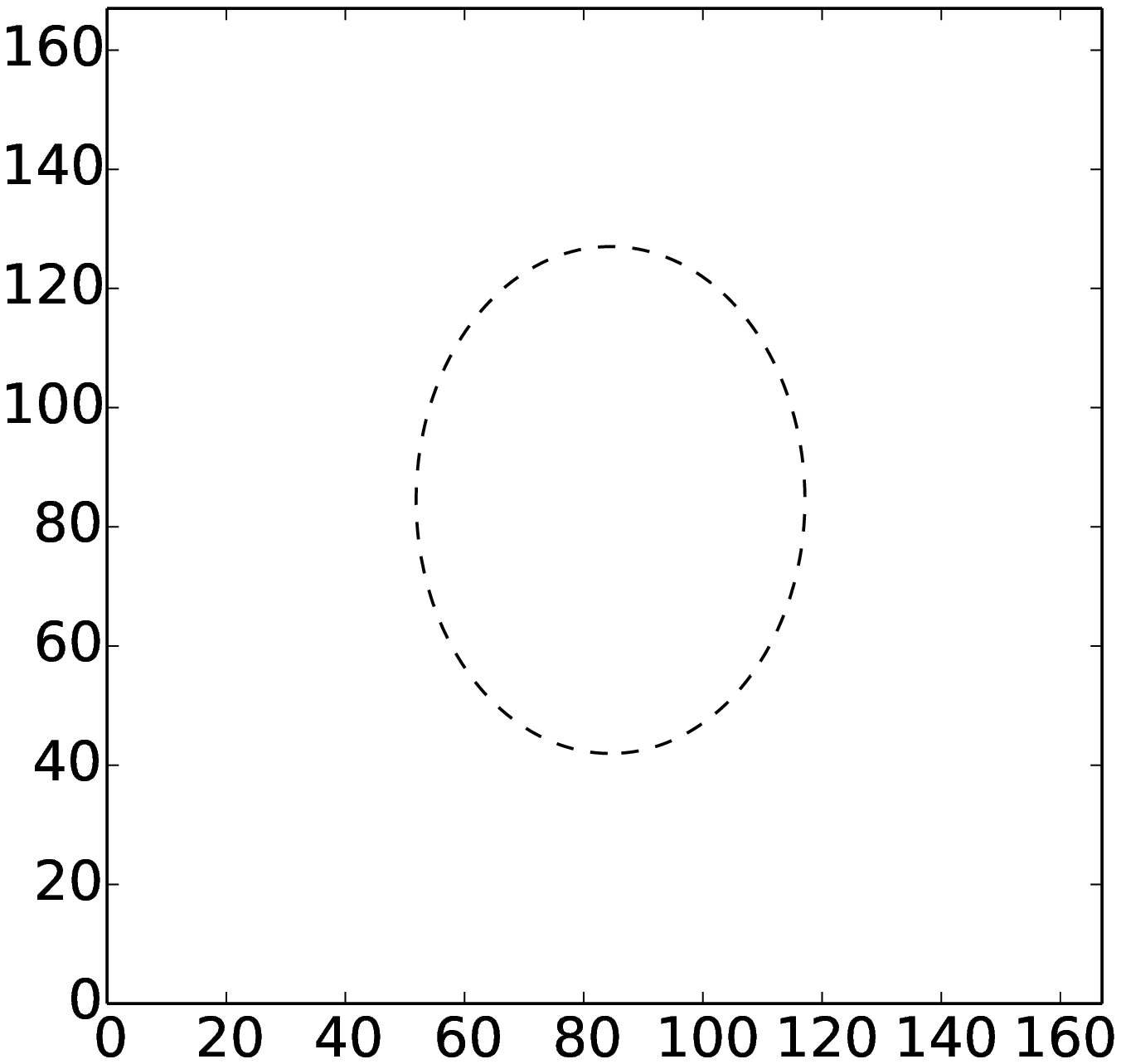} 
 \includegraphics[width=3.5cm]{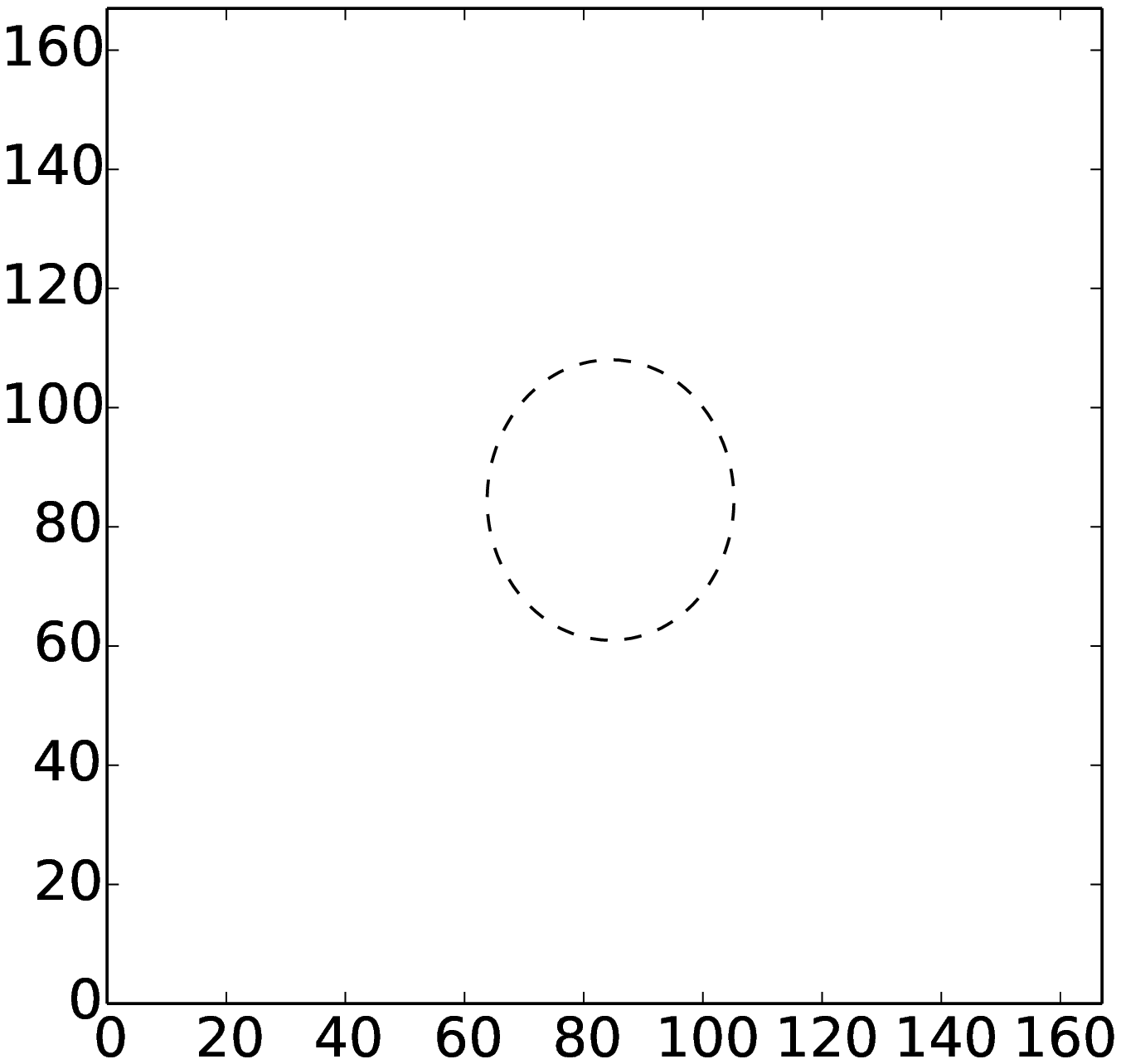}
 \includegraphics[width=3.5cm]{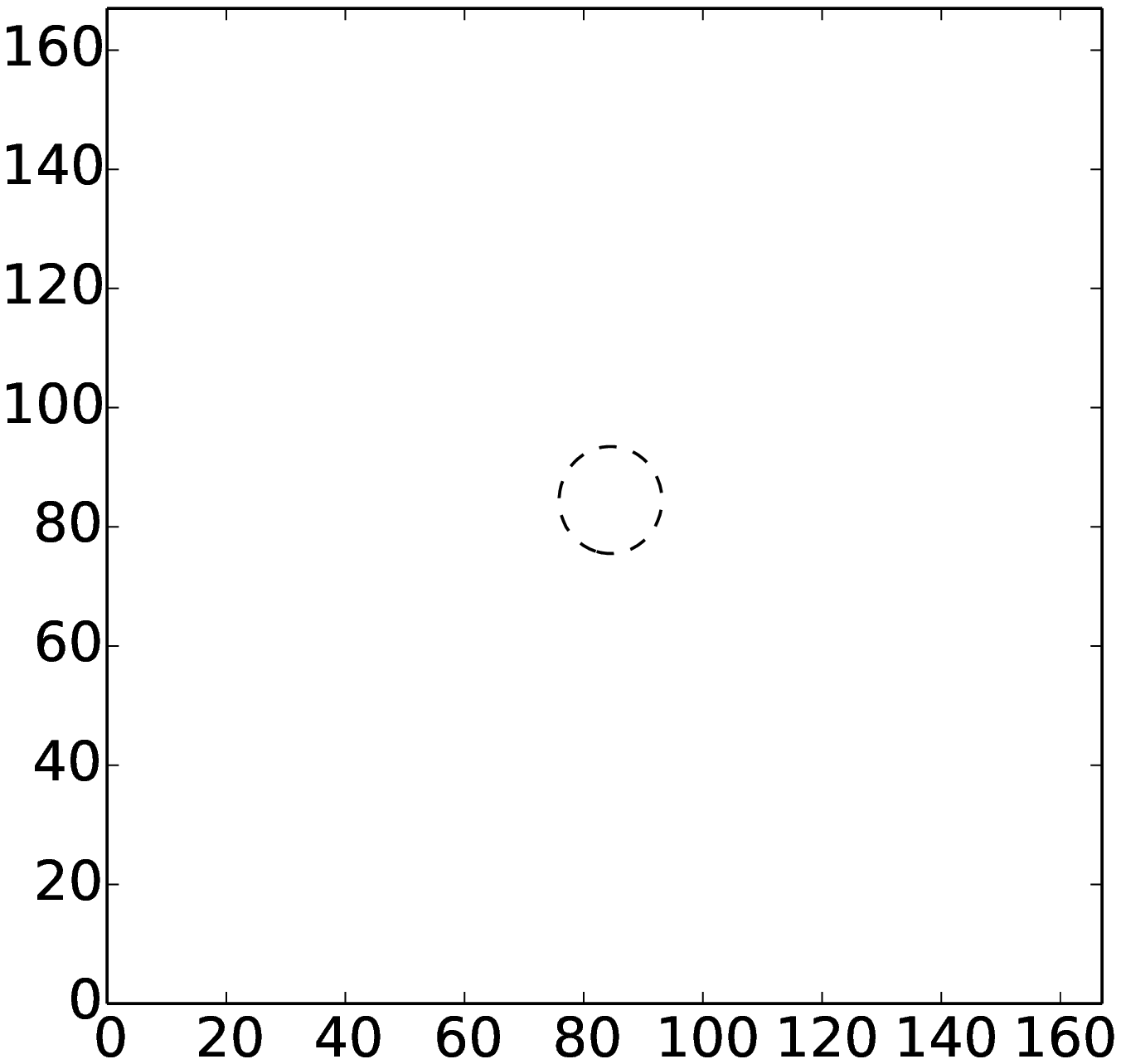}
 \caption{Ellipse-case: zero-level set (numerically -0.07-level) of $
    u^{\epsilon}$ with initial function $\tilde u_0$
 for $T=0$, $T=0.24$ and $T=0.36$.}
 \label{fig2}
 \end{figure}

\subsection{Discussion of the numerical results}

We recall that there are four numerical parameters, i.e. $h$, $\epsilon$, 
$r_0$ and $l_0$, in our approach. Incorporating the two space dimensions
we are confronted with a de facto five dimensional problem which we discretize.
Still, our best error is 0.0139 and in Table \ref{table5}, attained for the 
moderate parameter values $h=0.01$, $\Delta t = \epsilon^2=0.0004$, $r_0=100$
and $l_0=360$. We observed that the convergence properties improve when $\gamma$
is closer to $1$.

To classify our errors we compare with 
\cite{CarliniFalconeFerretti2010} where for the Kohn-Serfaty scheme of the 
curve shortening flow a $l_{\infty}$-error of $0.0078$ is obtained where 
$\Delta t = 0.02$ and $\Delta x = 0.0098$ , cf.
\cite[Table 2]{CarliniFalconeFerretti2010}. 
For the same flow the semi-Lagrangian scheme leads to a $l_{\infty}$-error
of $0.00832$, cf.
\cite[Table 2]{CarliniFalconeFerretti2010}. Note, that therein 
the initial function
is a higher order power of the difference $|x|^2-R_0^2$ than in our paper 
and this effects the error.
Furthermore, we refer to
\cite[Table 1]{CarliniFerretti2013}
where 
 a semi-Lagrangian approximation for our equation (\ref{81})
 in the case $\ga=\frac{1}{3}$ is considered.
For $\Delta x = 0.01$ and $\Delta t = 0.001$ the authors 
obtain there a $l_{\infty}$-error of $0.0542$ and of $0.0306$ for a modified 
scheme. 
Note, that the time interval for the calculation in each of the above cases 
is---as in our case---approximately 
$20\%$ of the time interval $[0, T^{*})$ where $T^{*}$ is the time
at which the evolving curve 'becomes a point'. Hence our Tables \ref{table5} and \ref{table6} contain values of
adequate size.





\end{document}